\documentclass[a4paper,12pt]{amsart}
\usepackage[utf8]{inputenc}
\usepackage{amsmath,amsfonts,amsthm,amssymb}

\makeatletter
\@namedef{subjclassname@2010}{%
  \textup{2010} Mathematics Subject Classification}
\makeatother

\newtheorem{Prop}{Proposition}
\newtheorem{Thm}{Theorem}

\newcommand{\pierwszapochodna}[2]{\frac{\partial #1}{\partial #2}}
\newcommand{\drugapochodna}[3]{\frac{\partial^2 #1}{\partial #2 \partial \bar{#3}}}
\newcommand{\poczasie}[2]{\frac{d^{#2} #1}{dt^{#2}}}

\newcommand{\Monge}[0]{Monge-Amp\`{e}re }
\newcommand{\Kahler}[0]{K\"{a}hler }

\begin{document}

\title{The geometry of $\mathbb{C}^2$ equipped with Warren's metric}
\author{Szymon Myga}
\address{Insitute of Mathematics\\ Jagiellonian University\\ 30-348 Kraków, Poland}
\email{szymon.myga@im.uj.edu.pl}

\begin{abstract}
The aim of this note is to describe the geometry of $\mathbb{C}^2$ equipped with a \Kahler metric defined by Warren. It is shown that with that metric $\mathbb{C}^2$ is a flat manifold. Explicit formulae for geodesics and volume of geodesic ball are also computed. Finally, a family of similar flat metrics is constructed.
\end{abstract}



\maketitle

\section{Introduction}

In his note~\cite{Warren} Warren showed an interesting construction of an entire solution to the unimodular complex \Monge equation in $\mathbb{C}^2$ that depends only on three real parameters. This solution is a special case of a non-convex solution to Donaldson's equation, i.e.
\[
 f_{tt}\,\Delta_x f - |\nabla_x f_t|^2 = 1,\qquad (t,x) \in \mathbb{R}\times\mathbb{R}^{n-1}.
\]
As showed by Donaldson in~\cite{Donaldson}, for $n=3$ a multiple of a solution to the above problem must also satisfy the unimodular complex Monge-Amp\`{e}re equation in complex dimension 2 with one `artificial' imaginary parameter. The equation comes from the study of the space of volume forms on compact Riemannian manifolds. It also has connections with free boundary problems and Nahm's equations from mathematical physics (see~\cite{Donaldson} for details). 

Some properties of solutions to this problem were studied by He in~\cite{He}, where a class of `trivial' solutions were exhibited, that is solutions of the form
\[
 f(t,x) = at^2 + tb(x) + g(x), \qquad (t,x) \in \mathbb{R}\times\mathbb{R}^{n-1}
\]
with $a = const.$, $\Delta b = 0$ and $\Delta g = (1 + |\nabla b|^2)/2a$.

Let $(z,w) = (x + iy, u + iv) \in \mathbb{C}^2$. Then Warren's solution is

\[
 f(z,w) = 4|z|^2e^{\text{Re}(w)} + e^{-\text{Re}(w)}.
\]

This plurisubharmonic function gives the \Kahler metric $g$ such that $\text{det}(g) = 1$. Explicitly
\[
g =
 \begin{pmatrix}
  \drugapochodna{f}{z}{z} & \drugapochodna{f}{z}{w} \\[0.5em]
  \drugapochodna{f}{w}{z} & \drugapochodna{f}{w}{w} 
 \end{pmatrix}
 =
 \begin{pmatrix}
  4e^u & 2\bar{z}e^u \\[0.5em]
  2ze^u & |z|^2e^u + \frac{1}{4}e^{-u} 
 \end{pmatrix},
\]

with inverse

\[
  \begin{pmatrix}
  \frac{1}{4}e^{-u} + |z|^2e^u & -2\bar{z}e^u \\[0.5em]
  -2ze^u & 4e^u 
 \end{pmatrix}.
\]

The construction of this metric was based on a method for constructing non-polynomial solution to the $k$-Hessian equations given by Warren in~\cite{Warren2}.

In the real case, it was proved by J\"{o}rgens, Calabi and Pogorelov (\cite{Jorgens},\cite{Calabi},\cite{Pogorelov}) that for $n\geq 2$, the only convex solutions to the unimodular \Monge equation on the whole $\mathbb{R}^n$ are quadratic functions. It is well known that similar property for entire plurisubharmonic solutions to the complex \Monge equation does not hold in $\mathbb{C}^n$. However, a question posed by Calabi in~\cite{Calabi1} whether a \Kahler metric given by the complex Hessian of such a solution is flat is still open. 

In this note the geometry of $(\mathbb{C}^2,g)$ is studied. First, we prove the following:
\begin{Thm}
 Warren's metric $g$ is flat.
\end{Thm}
So the Warren's metric is not a counterexample to Calabi's problem as was initially hoped. The proof is done by direct calculation of the curvature tensor in Section 2. In the following sections the geodesic equations are solved, allowing one to describe the geometry of geodesic balls. Finally, a family of similar flat metrics on $\mathbb{C}^2$ is constructed by slight generalization of Warren's argument.

\section{Curvature}

For any \Kahler metric $g$ the formulae for Christoffel symbols simplify (see~\cite{Ballmann}) to the
\[
 \Gamma^{\alpha}_{\beta \gamma} = \pierwszapochodna{g_{\gamma\bar{\nu}}}{z_{\beta}} g^{\bar{\nu}\alpha}, \;\;\; \forall \; \alpha, \beta, \gamma \in \{ 1,\dots, n \},
\]
where $g_{\gamma\bar{\nu}}$ are the coefficients of the metric and $g^{\bar{\nu}\alpha}$ are coefficients of the inverse. This makes the computation of Christoffel symbols for the metric $g$ straightforward:
\begin{align*}
 \Gamma^z_{zz} &= 0, \\
 \Gamma^z_{zw} &= \Gamma^z_{wz} = 2e^u\left(\frac{1}{4}e^{-u} +|z|^2e^u\right) + \bar{z}e^u(-2ze^u) = \frac{1}{2}, \\
 \Gamma^z_{ww} &= ze^u\left(\frac{e^{-u}}{4} + |z|^2e^u\right) -2ze^u\left(\frac{|z|^2e^u}{2} - \frac{e^{-u}}{8}\right) =\frac{z}{2}, \\
 \Gamma^w_{zz} &= 0, \\
 \Gamma^w_{zw} &= \Gamma^w_{wz} = (-2\bar{z}e^u)2e^u + 4e^u\bar{z}e^u = 0,\\
 \Gamma^w_{ww} &= (-2\bar{z}e^u)ze^t + 4e^u\left(\frac{|z|^2e^u}{2}-\frac{e^{-u}}{8}\right) = -\frac{1}{2}.
\end{align*}

The curvature tensor coefficients again simplify. For \Kahler metrics, the only non-zero coefficients can be
\[
 R^{\delta}_{\alpha\bar{\beta}\gamma} = -\pierwszapochodna{\Gamma^{\delta}_{\alpha\gamma}}{\bar{z}_{\beta}},\quad
 R^{\delta}_{\bar{\alpha}\beta\gamma} = \pierwszapochodna{\Gamma^{\delta}_{\beta\gamma}}{\bar{z}_{\alpha}}, \quad \forall \; \alpha, \beta, \gamma \in \{1,\ldots,n\}
\]
and their conjugates. So, clearly the curvature of $g$ vanishes.

\section{Geodesics and Incompleteness}

With the Christoffel symbols computed, the geodesic equations in complex coordinates are

\begin{equation*}
\poczasie{z}{2} + \poczasie{z}{ }\poczasie{w}{ } + \frac{z}{2}\left(\poczasie{w}{}\right)^2 = 0
\end{equation*}
\begin{equation}\label{eq:drugie-geodezyjne}
\poczasie{w}{2} - \frac{1}{2}\left(\poczasie{w}{}\right)^2 = 0. 
\end{equation}

The order of the first equation can be reduced by plugging the second equation into it, thus producing
\[
 \poczasie{z}{} + z\poczasie{w}{} = C,
\]
where $C$ is a constant.
Written in real coordinates equation (\ref{eq:drugie-geodezyjne}) is
\begin{align*}
 u'' &= \frac{(u')^2 - (v')^2}{2}, \\
 v'' &= u'v'.
\end{align*}

It can be solved given an initial velocity vector $(U,V)$ at a starting point $(u_0,v_0)$. The second equation is a derivative of the logarithm, so it reduces to

\[
 v' = Ve^{u-u_0}.
\]
Now, after putting this into the first equation and making a substitution $\alpha(u) = (u')^2$ one gets

\[
  u' =  \pm\sqrt{(U^2 + V^2)e^{u-u_0} - V^2e^{2u - 2u_0}}, \\
\]
where the sign depends on the sign of the initial velocity component $U$ and time. Since this equation is autonomous it integrates to
\[
 u(t) = u_0 - \log\left(\frac{U^2 + V^2}{4}t^2 -Ut + 1\right),
\]
but it might not be defined for every $t > 0$. If $U \leq 0$, then the formula holds for every $t > 0$. For $U >0, V\neq0$ the formula also holds but this time the velocity $u'$ switches sign at time $t_0 = 2U/(U^2 + V^2)$ since then $u'$ vanishes and $u''$ is negative. For $U>0$ and $V=0$ the curve $u(t)$ reaches infinity at $t_0$, independently of the initial point. This shows incompleteness of the metric $g$, already established by Warren in~\cite{Warren}.

Equipped with an explicit fomula for $u$, one integrates $v'$ to
\[
 v(t) = 2\arctan\left(\frac{U}{V}\right)\; -2\arctan\left(\frac{U - \frac{U^2 + V^2}{2}t}{V}\,\right) + v_0,
\]
provided $V\neq 0$, otherwise $v(t) = v_0$.
\vspace{1em}

Now, the equation for the first component of the geodesic is linear
\[
 \poczasie{z}{} + z\poczasie{w}{} = C
\]
for some constant $C$ depending on the initial velocity. The solution then is
\[
 z(t) = e^{-w}\left[(Z + z_0W)\int_0^t e^{w} \;+ z_0e^{w_0}\right],
\]
Where $Z = X + iY, W = U + iV$ are the initial velocities and $z_0, w_0$ are the starting points. Denote by $Q(t) = \frac{|W|^2}{4}t^2 - Ut + 1 = e^{u_0-u}$ and $\vartheta(t) = v(t) - v_0$, then
\begin{align*}
 z(t) =& (Z + z_0W)[Q(t)\cos\vartheta(t) - iQ(t)\sin\vartheta(t)]\int_0^t \frac{\cos\vartheta(s) + i\sin\vartheta(s)}{Q(s)}ds \\
 & + z_0[Q(t)\cos\vartheta(t) - iQ(t)\sin\vartheta(t)].
\end{align*}
Then formulae for the sine and the cosine of the arctangent yield
\begin{align*}
 \cos\vartheta(s) &= Q^{-1}(s)\left(\frac{U^2 - V^2}{4}s^2 -Us + 1\right), \\
 \sin\vartheta(s) &= VQ^{-1}(s)\left(s - \frac{U}{2}s^2\right),
\end{align*}
so the integral is
\[
 \int_0^t e^{w-w_0} = Q^{-1}(t)\left[t - \frac{U}{2}t^2 + i \frac{V}{2}t^2 \right] = Q^{-1}(t)\left[1 - \frac{\bar{W}}{2}t^2\right].
\]

Finally,
\begin{align*}
 z(t) &= \left[\frac{W^2}{4}t^2 - Wt + 1\right]\left[(Z + z_0W)Q^{-1}(t)\left(1+ \frac{\bar{W}}{2}t^2\right) + z_0\right] \\
 &= [Z + z_0W]\left[t-\frac{W}{2}t^2\right] + z_0\left[\frac{W^2}{4}t^2 + 1\right] \\
 &= Z\left[t-\frac{W}{2}t^2\right] + z_0\left[-\frac{W^2}{4}t^2 + 1\right] = \left[1-\frac{W}{2}t\right]\left[z_0 + \left(Z + \frac{z_0W}{2}\right)t\right].
\end{align*}

\section{Geometry of geodesic balls}

Since the metric is flat, one knows by the Frobenius theorem that around each point $z_1 \in (\mathbb{C}^2,g)$ there is a neighbourhood $U_{z_1}$ and a map $\varphi_{z_1} :U_{z_1} \rightarrow (\mathbb{C}^2,\mathcal{E})$ ($\mathcal{E}$ is the usual Euclidean metric) that is an isometry onto its image. Since $\text{det}(g) =1$, those two facts imply that for any $z_1$, the volume of geodesic ball $B_g(z_1,r)$ is the same as the volume of Euclidean ball of radius $r$, for $r$ sufficiently small.

In the case of the metric $g$ more is true. Let $\exp_g(z_1,\cdot)$ denote the exponential map from the metric ball $\widetilde{B}(z_1,r) = \{Z \in T_{z_1}\mathbb{C}^2\; |\; g_{z_1}(Z,Z) < r^2\}$ to the geodesic ball $B_g(z_1,r)$ for the fixed, small radius $r$. Now if ${}^t g^{-1/2}(z_1,\cdot)$ is the transpose of the positive square root of matrix $g^{-1}$ at a point $z_1$ then $\widetilde{B}(z_1,r) = {}^t g^{-1/2}(z_1,B_{\mathcal{E}}(0,r))$ and since $\det(g) = 1$, the same is true for ${}^t g^{-1/2}$. So this reduces the volume of the ball $B_g(z_1,r)$ to the following formula 
\[
 \int_{B_g}d\,vol_g = \int_{B_{\mathcal{E}}} \text{jac}(\exp_g(z_1,{}^t g^{-1/2}(x))) d\,vol_{\mathcal{E}},
\]
and the Jacobian can be computed, since the formulae for the geodesics are given explicitly. Namely, if $(e_1, e_2)$ is the canonical basis of $\mathbb{C}^2$, then in new variables $X + iY = Z = {}^t g^{-1/2}(z,e_1); U + iV = W = {}^t g^{-1/2}(z,e_2)$ the equation for metric ball at $z$ becomes $|Z|^2 + |W|^2 \leq r^2$. Computing the Jacobian of the exponential map in this variables reduces to computing the derivatives of geodesics with respect to variables $(X,Y,U,V)$ at time $t=1$ thus revealing that

\begin{equation*}
 \text{det}
 \begin{pmatrix}
  \dfrac{2-U}{2} & \dfrac{V}{2} & . & . \\[0.7em]
  \dfrac{-V}{2} & \dfrac{2-U}{2} & . &. \\[0.7em]
  0 & 0 & \dfrac{2-U}{2Q(1)} & \dfrac{-V}{2Q(1)} \\[0.7em]
  0 & 0 & \dfrac{V}{2Q(1)} & \dfrac{2-U}{2Q(1)}
 \end{pmatrix}
 = 1.
\end{equation*}

Thus for $U<2$ the geodesic ball has the same volume as the Euclidean ball. At $U=2$ the entries in the matrix above are undefined for $V=0$, since then $Q(1) = 0$. 

That means the set of directions $N$ for which the geodesics reach infinity in finite time is negligible, namely  $N= \{(X,Y,U,0) \in T\mathbb{C}^2 \;|\;\; U > 0\}$. Any geodesic with intial velocity $V \neq 0$ or $U \leq 0$ can be extended indefinitely from any initial point and the set $B_g(\cdot,r)\setminus N$ will have the same volume as eucildean $r$-ball for any $r$. Here the set $B_g(\cdot,r)\setminus N$ is understood as set of constant ($\leq r$) speed geodesics with the same initial point and with initial velocities in complement of $N$.

\section{Generalization}

The orginal approach of Warren was based on the problem of finding non-polynomial solutions to the $k$-Hessian equations. There, given $(x,y) \in \mathbb{R}^{n-1}\times\mathbb{R}$ one was looking for a function $h(y)$, such that the function $f(x,y) = |x|^2e^y + h(y)$ satisfies
\[
 \sigma_k(D^2f) = 1, \qquad 2k-1\leq n,
\]
with $\sigma_k$ being the $k$-th elementary symmetric polynomial, that is the sum of determinants of all principal minors of size $k \times k$. For $n=3$ and $k=2$ with $h(y) = e^{-y}/4$ this gave a solution to Donaldson's equation and the metric $g$ under consideration. The aim of this section is to generalize the construction of $g$. This is done by noticing that leaving the quadratic part in the potential $f$ and changing $e^y$ and $h(y)$ can still produce flat metrics with unit determinant under suitable assumptions.

\begin{Prop}
 Let $e^{-h(w)}dw\,d\bar{w}$ be the a flat metric on $\mathbb{C}$, then the metric $K_h$ defined as
 \[
 K_h =
  \begin{pmatrix}
  e^{h(w)} & \bar{z}\dfrac{\partial e^{h(w)}}{\partial\bar{w}}\\[0.9em]
  z\dfrac{\partial e^{h(w)}}{\partial w} & |z|^2\dfrac{\partial^2 e^{h(w)}}{\partial w\partial\bar{w}} + e^{-h(w)}
 \end{pmatrix}
 \]
 is a flat \Kahler metric on $\mathbb{C}^2$, such that $\textup{det}K_h = 1$
\end{Prop}
\begin{proof}
 The curvature of $e^{-h(w)}$ is
 \[
  R(e^{-h}) = e^h\left(\drugapochodna{e^{-h}}{w}{w} - e^h\left|\pierwszapochodna{e^{-h}}{w}\right|\right) = \Delta h.
 \]
 So flatness is equivalent to the harmonicity of $h$. Since the metric is given explicitly the proof of K\"{a}hlerness follows directly from computation. Similarly for the determinant:
 \[
  \text{det}(K_h) = 1 + |z|^2\left(e^h\drugapochodna{e^h}{w}{w} - \left|\pierwszapochodna{e^h}{w}\right|^2\right) = 1 + |z|^2(e^{2h}R(e^h)) = 1,
 \]
 where $R(e^h)$ is also zero, since both $h$ and $-h$ are harmonic.
 
 Let $\Gamma^\alpha_{\beta\gamma}$ denote the Christoffel symbols of $K_h$. Except for $\Gamma^z_{ww}$ and $\Gamma^w_{ww}$ the computation is straightforward, showing that every symbol is holomorphic. For the remaining two one needs to use the harmonicity of $h$, getting
 \begin{align*}
  \Gamma^z_{ww} &= ze^{h}\left(\left(\pierwszapochodna{h}{w}\right)^2 + \frac{\partial^2 h}{\partial w^2}\right)\left(|z|^2e^{h}\left|\pierwszapochodna{h}{w}\right|^2 + e^{-h}\right) - \\  
  &\left(|z|^2e^{h}\pierwszapochodna{h}{\bar{w}}\left(\frac{\partial^2 h}{\partial w^2} + \left(\pierwszapochodna{h}{w}\right)^2\right) - e^{-h}\pierwszapochodna{h}{w}\right)\left(ze^{h}\pierwszapochodna{h}{w}\right) = \\
  &z\left(2\left(\pierwszapochodna{h}{w}\right)^2 + \frac{\partial^2 h}{\partial w^2}\right), \\
  \Gamma^w_{ww} &= ze^{h}\left(\left(\pierwszapochodna{h}{w}\right)^2 + \frac{\partial^2 h}{\partial w^2}\right)\left(-\bar{z}e^{h}\pierwszapochodna{h}{\bar{w}}\right) + \\
  &\left(|z|^2e^{h}\pierwszapochodna{h}{\bar{w}}\left(\frac{\partial^2 h}{\partial w^2} + \left(\pierwszapochodna{h}{w}\right)^2\right) - e^{-h}\pierwszapochodna{h}{w}\right)e^{h} = -\pierwszapochodna{h}{w}.
 \end{align*}
 Both symbols are holomorphic and thus $K_h$ is flat.
\end{proof}

\vspace{1em}
\textbf{Acknowledgement.} The author would like to thank Sławomir Dinew for his advice and patience.
The research was supported by NCN grant 2013/08/A/ST1/00312.

\end{document}